\documentclass[12pt]{article}

\usepackage{geometry}
\usepackage{graphicx}
\usepackage{amssymb}
\usepackage{amsmath,tikz-cd}
\usepackage{mathtools}
\usepackage{bbm} % For blackboard bold digits
\usepackage{tikz}
\usetikzlibrary{matrix,arrows,decorations.pathmorphing}

\newcommand{\R}{\mathord{\mathbb{R}}}
\newcommand{\N}{\mathord{\mathbb{N}}}

\newcommand{\B}{\mathord{\mathcal{B}}}

\newcommand{\myimplies}{\mathord{\;\Rightarrow\;}}

\newcommand{\down}{\mathord{\downarrow}}
\newcommand{\up}{\mathord{\uparrow}}

\newcommand{\ideals}{\mathop{\mathrm{Id}}}
\newcommand{\cpt}{\mathop{\mathrm{cpt}}}
\newcommand{\tbigvee}{\textstyle{\bigvee\,}}
\newcommand{\tbigcup}{\textstyle{\bigcup\,}}

\newcommand{\Irr}{\mathord{\mathrm{Irr}\kern 2pt}}
\newcommand{\Spec}{\mathop{\mathrm{Spec}}}

\newcommand{\SpecM}{\mathop{\mathrm{Spec_{Max}}}}

\newcommand{\maxx}{\mathop{\mathrm{max}}}

\newcommand{\llangle}{\mathord{\langle\kern-3pt\langle}}
\newcommand{\rrangle}{\mathord{\rangle\kern-3pt\rangle}}
\newcommand{\coz}{\mathop{\mathrm{coz}}}
\newcommand{\spec}{\mathop{\mathrm{spec}}}

\usepackage{stackengine}

\usepackage{amsthm}
\theoremstyle{plain}
\newtheorem{thm}{Theorem}[subsection]
\newtheorem{lem}[thm]{Lemma}
\newtheorem{prop}[thm]{Proposition}
\newtheorem{cor}[thm]{Corollary}

\theoremstyle{definition}
\newtheorem{defn}[thm]{Definition}

\newtheorem{exmp}[thm]{Example}

\theoremstyle{remark}

\newtheorem*{rem}{Remark}

\title{Conjunctive Join-Semilattices}

\author{
  Delzell, Charles N.\\
  \texttt{delzell@math.lsu.edu }
  \and
  Madden, James J.\\
  \texttt{mmmad@lsu.edu}
  \and
  Ighedo, Oghenetega\\
  \texttt{ighedo@unisa.ac.za}
}

\begin{document}
\maketitle
\section{Introduction}

\iffalse
The introduction should contain 
\begin{enumerate}
\item known representation theorems for join semilattices: distributive case;
\item some history on the congruence: Glivenko, Frink, Pierce; 
\item a discussion of the results of this paper and their applications
\end{enumerate}
\fi

One of the most general extensions of Stone's Representation Theorem for Boolean Algebras  concerns distributive join-semilattices, i.e., join-semilattices in which any finite cover of any element has a refinement that exactly covers that element.  Distributivity thus defined assures that a join-semilattice has sufficiently many prime ideals to represent it as the join-semilattice of all compact opens in some sober $T_0$-space in which every open is a union of compact opens.  See \cite{G}, Chapter II, section 5. 

The present paper includes a representation theorem for join-semilattices that need not be distributive but instead (or in addition) are \textit{conjunctive}.  A join-semilattice $L$ is said to be conjunctive (or have the conjunction property) if it has a top element $1$ and it satisfies the following first-order condition: for any two distinct $a,b\in L$, there is $c\in L$ such that either $a\vee c\not=1=b\vee c$ or $a\vee c=1\not=b\vee c$.  Equivalently, one can show, a join-semilattice is conjunctive if every principal ideal is an intersection of maximal ideals.   We present simple examples showing that a conjunctive join-semilattice may fail to have any prime ideals.  We show that every conjunctive join-semilattice is isomorphic to a join-closed subbase for a compact $T_1$-topology on $\maxx L$, the set of maximal ideals of $L$.  The representation is canonical, in that when applied to a join-closed subbase for a compact $T_1$-space $X$, the space produced by the representation is homeomorphic with $X$.  

 For any ideal $I\subseteq L$ there is a join-semilattice morphism $\phi:L\to M$ and a maximal ideal $J$ of $M$ such that $I=\phi^{-1}(J)$.   We say a  join-semilattice morphism $\phi:L\to M$ is  \textit{conjunctive} if $\phi^{-1}(w)$ is an intersection of maximal ideals of $L$ whenever $w$ is a maximal ideal of $M$.  We show that every conjunctive morphism between conjunctive join-semilattices is induced by a multi-valued function from $\maxx M$ to $\maxx L$.  Thus, the representation is functorial in a manner that is reminiscent of distributive lattices or rings, but with some unexpected and intriguing modifications.

The conjunction property---or more precisely its dual, the ``disjunction property''---first appeared in Wallman's 1938 paper \cite{W} as a condition on a  distributive lattice that guarantees the existence of sufficiently many maximal filters to distinguish between elements of the lattice.  This property was generalized by Pierce in his 1954 paper \cite{P}, in a way that is meaningful in any semigroup.  Simmons and Macnab \cite{S} (1978) seems to have been the first to use of the term ``conjunctive'' with the same meaning as in the present paper, but they applied it to distributive lattices, not to join-semilattices in general.  Johnstone mentions their work and uses their terminology in \cite{J84b}.   The conjunction property  is closely related to the property of subfitness, introduced in 1973 by Isbell \cite{I73} as a separation axiom for locales.  Subfitness is discuused at length in Chapter V of the book \cite{PP}, but here again, though the definition is the same, it is applied only to frames (i.e., complete lattices in which $\wedge$ distributes over all suprema).   

A conjunctive join-semilattice must have a top element, but  applications (such as those discussed in \cite{MZ}) require more generality.  We define a new property, \textit{ideal conjunctivity}, that enables us to extend the representation theory to join semilattices that do not have $1$.  Ideal conjunctivity is a generalization to join-semilattices of the notion of ``finite subfitness'' that was introduced in \cite{MZ}; see subsection \ref{mzsect}, below. 

Every join semilattice $L$ with $1$ has a smallest congruence (called $R^1(L)$) for which the quotient is conjunctive, and $R^1(L)$ is the maximal congruence in which the congruence class of $1$ is a singleton.  (This is easily deduced from results in \cite{P}.)  This situation carries over to distributive lattices and to compact frames, since $R^1(L)$ respects meets and infinite joins in these settings.  The study of this congruence in the context of point-free topology was initiated by Johnstone in \cite{J84}, where he used it to modify the lattice of ideals of a bounded distributive lattice $L$ to produce the locale of ``almost maximal ideals'' of $L$.  As is commonly done in point-free topology, Johnstone used a nucleus to represent the congruence.  (Recently,  Haykazyan [H] has studied this further.)   Banaschewski and Harting \cite{B85} considered the congruence for general compact frames, calling the quotient mapping the ``saturation,'' and Banaschewski continued the study of the saturation in \cite{B02}.      A manifestation of Pierce's theorem in frames is stated in Corollaries V.1.3.3  and V.1.4.1 of \cite{PP}. 

Here is an outline of the present paper.   We have included some expository material, especially in subsections 2.3 and 4.1, to provide continuity between our work and background literature and to fill in a few details that are not addressed in other sources. 

Section 2 begins with a review of basic definitions.  An ideal of a join-semilattice is a $\vee$-closed downset; a prime ideal is a proper ideal whose complement is a filter.  We provide an example of a finite join-semilattice with three maximal ideals and no prime ideals.  After this, we review facts about the complete lattice of all ideals of a join-semilattice.  In general, it may fail to satisfy any distributive laws, but it may be characterized as the solution to a universal mapping problem.   In the third subsection, we present Pierce's theory \cite{P} specialized to join-semilattices.  Finally, we introduce and discuss the new concept of \textit{ideal conjunctivity}, which generalizes the concept of conjunctivity to join-semilattices without top element. 

In Section 3, we prove the Representation Theorem for Conjunctive Join Semilattices and its partial generalization to ideally conjunctive join semilattices.  Then we apply the theory to prove a far-reaching generalization of a result of Martinez and Zenk concerning ``Yosida frames.''  

In Section 4, we review the representation theory for distributive join-semilattices, then show that a \textit{complete} conjunctive join-semilattice is distributive if and only if all its maximal ideals are prime.  It is an open question as to whether this is true without the completeness hypothesis.  

The final section considers two problems that relate to distributive lattices.  A base for a topological space is said to be annular if it is a lattice.  A Wallman base for a space $X$ is an annular base such that for any point $u$ in any basic open $U$, there a basic open $V$ that misses $u$ and together with $U$ covers $X$.  It is easy to show that every Wallman base is conjunctive.  We give an example of a conjunctive annular base that is not Wallman.  Finally, we examine the free distributive lattice $dL$ over a conjunctive join semilattice $L$.  In general, it is not conjunctive, but we show that $dL/R^1(dL)$ is isomorphic to 
$wL: =$ the sub-lattice of the topology of the representation space that is generated by $L$.  The passage from $L$ to $wL$ is not functorial.

The present paper demonstrates that the category of conjunctive join-semilattices and conjunctive morphisms has a rich and interesting theory and that much of the existing theory of conjunctive distributive lattices and subfit frames springs from the properties of these more-elementary structures.  The foundational role of join-semilattices in the theory of frames and locales was highlighted in \cite{JT}, where the authors suggested the following fruitful analogy:
$$\text{ frames : rings :\kern 1.5pt: complete join-semilattices : abelain groups.}$$ Few authors, however, have built on this.  We hope that our work will inspire others to explore the role of conjunctive join-semilattices in pointfree topology and other areas.  We have included a number of unanswered questions that appear at the ends of Sections 1 through 4.

\section{Join-Semilattices}

%%%%%%%%%%%%%%%%%%%%%%%%
%%%%%%%%%%%%%%%%%%%%%%%%
\subsection{Basic Facts}
%%%%%%%%%%%%%%%%%%%%%%%%
%%%%%%%%%%%%%%%%%%%%%%%%

A \textit{join-semilattice} is a set $L$ equipped with  an associative, commutative, idempotent binary operation  $\vee$.  $L$ is partially ordered by the relation $x\leq y$, which by definition means $x\vee y = y$.  In this order, $x\vee y$ the least upper bound of $x$ and $y$.   The largest (respectively, smallest) element of $L$, if it contains one, is denoted by $1$ or $1_L$ (respectively, $0$ or $0_L$).  We say $U\subset L$ is an \textit{up-set} if $a\in U$ and $a\leq b$ implies $b\in U$.  The up-set $\{\,b\in L\mid a\leq b\,\}$ is denoted by $\up a$.  Down-sets and $\down a$ are defined analogously.

\begin{defn} Let $L$ be a join-semilattice.  
\begin{itemize}\item[$(i)$]  We call a subset $I\subset L$ a \textit{join-semilattice-ideal} (or simply an \textit{ideal} when the context is clear) if it is a down-set and $a\vee b\in I$ whenever $a,b\in I$.

\item[$(ii)$]  We call an up-set a \textit{filter} if it is non-empty and contains a lower bound for any two of its elements.  We say an ideal $I\subseteq L$ is \textit{prime} if it is non-empty and its complement in $L$ is a filter. 

\item[$(iii)$]   If $a\in L$, then an ideal of $L$ is said to be  \textit{$a$-maximal} if it does not contain $a$, and any properly larger ideal does contain $a$.  When $L$ has a top element $1$, a $1$-maximal ideal is  called simply \textit{maximal}.  When $L$ lacks a top element, a proper ideal that is contained in no larger proper ideal is said to be \textit{maximal proper}.
\end{itemize}\end{defn}

\begin{lem}\label{lemlem} Suppose $L$ is a join-semilattice and $a\in L$.  Every proper ideal of $L$ that does not contain $a$ is contained in an $a$-maximal ideal.\end{lem}
\begin{proof} Suppose $\mathcal C$ is an increasing chain of ideals of $L$, none of which contains $a$.  Evidently $\bigcup\mathcal C$ does not contain $a$, and it is a downset closed under $\vee$, so it is an ideal.  This shows that every increasing chain of ideals not containing $a$ is contained in an ideal not containing $a$. The lemma therefore follows from Zorn's Lemma.\end{proof} 

\begin{exmp}\label{maxnotprime}\textit{A maximal ideal of a join-semilattice need not be prime.}  Let $X$ be the set $\{x,y,z\}$, and let $L =\{xyz, xy, xz, yz, \emptyset\}\subseteq \mathcal P X$, where we abbreviate the subset $\{x,y\}$ as $xy$.  The maximal ideals of $L$ are $m_z:=\{xy, \emptyset\}$, $m_y:=\{xz, \emptyset\}$, and $m_x:=\{yz, \emptyset\}$.  Observe that $L\setminus m_z = \{xyz, xz, yz\}$, and note that $xz$ and $yz$, though not in $m_z$, have no lower bound in $L\setminus m_z$.  Thus, $m_z$ is not prime.  Similarly, neither are the other maximal ideals.   The only proper ideal of $L$ other than the three maximal ideals is $\{\emptyset\}$, which is also not prime since its complement is not a filter; the improper ideals $L$ and $\emptyset$ are of course not prime.  Thus, while every nonempty join-semilattice must have at least one
maximal ideal (by Lemma \ref{lemlem}), it need not have any prime ideals.. \end{exmp}

Let $\mathbbm 2:=\{0,1\}$ be the join semilattice with $0\vee 1=1$.  For any join-semilattice $L$ and any ideal $I\subseteq L$, define $\phi_I:L\to \mathbbm 2$ by setting $\phi_I(a) = 0$ if $a\in I$ and $\phi_I(a) = 1$ if $a\not\in I$.  Then $\phi_I$ preserves $\vee$ by the definition of ideal.  In contrast, if $L$ is a lattice and $I$ is proper, then $\phi_I$ is a lattice morphism (i.e., preserves both $\vee$ and $\wedge$) if and only if $I$ is prime. 

%%%%%%%%%%%%%%%%%%
\subsection{The Join-Semilattice of Ideals $\ideals L$}
%%%%%%%%%%%%%%%%%%

\textit{Throughout the remainder of this paper, $L$ denotes a join-semilattice.  Sometimes additional conditions are imposed.}  The set of all ideals of $L$, including the improper ideals $\emptyset$ and $L$, is denoted by  $\ideals L$.   Since any intersection of ideals is an ideal, $\ideals L$ is a complete lattice (with meet being set-theoretic intersection and join being the meet of all upper bounds).  For any subset $X\subset L$, the intersection of all ideals containing $X$ is denoted by $\langle X\rangle$.  Note that $\langle\{x\}\rangle = \down x$.  Evidently, 
$$\langle X\rangle =\{\,y\in L\mid y\leq \tbigvee X'\;\text{for some finite}\; X'\subseteq X\,\}.$$   Observe that $\ideals L$, though it is a lattice, does not generally satisfy any distributive laws.   The equational law: $$\text{for all $a\in \ideals L$ and all $B\subseteq \ideals L$, $a\wedge \bigvee\! B=\bigvee\{\,a\wedge b\mid b\in B\,\}$}$$ holds if and only if $L$ satisfies the  \textit{distributive axiom for join-semilattices}---see Section~\ref{djsl}.

The injection map $\down: L\to \ideals L$ is a join-semilattice morphism since $\down a \vee \down b = \langle \down a \cup \down b\rangle = \down(a\vee b)$.  To simplify notation, we sometimes identify $L$ with its image in $\ideals L$.  For example, if $I\in \ideals L$ and $a\in L$, $I\vee a$ is understood to mean $\langle I\cup\{a\}\rangle$.

\begin{lem}\label{univdown} Let $J$ be a complete join-semilattice. If $f:L\to J$ is a $\vee$-morphism,  then there is a unique morphism $\overline f: \ideals L\to J$ that preserves all suprema and satisfies $\overline f\circ\down = f$.   \end{lem}

\begin{proof}  Define $\overline f:\ideals L\to J$ by  $\overline f(I):=\tbigvee\{\,f(y)\mid y\in I\,\}$.  If $\mathcal A\subseteq \ideals L$, then 
\begin{align*}
\overline f(\tbigvee \mathcal A) 
&= \tbigvee\{\,f(y)\mid y\leq \tbigvee\negthinspace A,\;A\;\text{a finite subset of }\tbigcup\mathcal A \,\}\\
&= \tbigvee\{\,f(y)\mid y\in I,\; I\in \mathcal A\,\},\;\text{since $f$ preserves finite $\vee$s},\\
&= \tbigvee\{\,f(I)\mid I\in \mathcal A\,\}
\end{align*}
For any $a\in L$, $\overline f(\down a) = \bigvee\{\,f(y)\mid y\leq a\,\} = f(a)$.  Thus $\overline f\circ\down= f$.
\end{proof}

Recall that an element $c$ of a complete join semilattice is said to be \textit{compact} if: whenever $c\leq \bigvee X$ for some subset $X$ of the lattice, there is a finite subset $X'$ of $X$ such that $c\leq \bigvee X'$.  A join-semilattice is said to be \textit{algebraic} if it is complete (and hence has $0$ and $1$) and is generated by its compact elements.  The following facts are well-known; see \cite{B}, VIII.5.

\begin{enumerate}
\item[$(i)$] Let $K$ be a complete join-semilattice. The set of compact elements, denoted by $\cpt K$, forms a sub-join-semilattice of $K$ containing $0$.  For any element $a\in K$, let $C(a):=\{\,c\in \cpt K\mid c\leq a\,\}$. Then $C(a)$ is an ideal of $\cpt K$, and the map $C:K\to \ideals\cpt K$ is order-preserving.  For any ideal $I\subseteq \cpt K$, we have $\bigvee I \in K$ since $K$ is complete.  The map $\bigvee:\ideals\cpt K\to K$ is order-preserving.  
\item[$(ii)$]  For any join semi-lattice $L$, $I\in \ideals L$ is compact if and only if $I = \emptyset$ or $I = \down a$ for some $a\in L$.   Since every ideal is the supremum of the principal ideals in it,  $\ideals L$ is algebraic. 
\item[$(iii)$] If $A$ is an algebraic join-semilattice, then $\bigvee C(a) = a$ for all $a\in A$ and $C(\bigvee I) = I$ for all $I\in \cpt A$.  Hence, the maps $C$ and $\bigvee$  are inverses of one another, giving an order-isomorphism $A\cong \ideals\cpt A$.
\end{enumerate}

%%%%%%%%%%%%%%%%%%%%%%%%
%%%%%%%%%%%%%%%%%%%%%%%%
\subsection{Pierce Congruences and Conjunctivity}
%%%%%%%%%%%%%%%%%%%%%%%%
%%%%%%%%%%%%%%%%%%%%%%%%

In this subsection, we present results of \cite{P}, specialized to join semilattices.  We include proofs because they do not take up much space, and it is useful to have them at hand in a notation that is consistent with the rest of this paper.

We say $R\subseteq L\times L$ is a $\vee$-congruence on $L$ if $R$ is an equivalence relation and for all $a,a',b\in L$, $(a, a')\in R$ implies $(a\vee b, a'\vee b)\in R$.  The set of equivalence classes of $R$ is denoted by $L/R$, and the class of $a\in L$ is denoted by $[a]$, or $[a]_R$ if reference to $R$ is needed.  The rule $[a]\vee [b]:=[a\vee b]$ defines an operation on $L/R$, since if $(a,a'), (b,b') \in R$, then $a\vee b\sim a'\vee b\sim a'\vee b'$.  With this operation, $L/R$ is a join-semilattice, and $a\mapsto [a]:L\to L/R$ is a surjective semilattice morphism.   If $R,R'$ are $\vee$-congruences on $L$, we say \textit{$R$ is weaker than $R'$} or  \textit{$R'$ is stronger than $R$} if $R\subseteq R'$.  The weakest $\vee$-congruence on $L$ is equality and the strongest is $L\times L$.

Let $Y$ be a subset of $L$.  We define the relation $R^Y=R^Y(L)$ by:
$$(a,a')\in {R}^Y\quad\Leftrightarrow_{def}\quad \forall x\in L,\, x\vee a\in Y \Leftrightarrow  x\vee a'\in Y.$$     We define $$(Y:a):=\{\,x\in L\mid x \vee a \in Y\,\}.$$ We call the elements of $(Y:a)$ the \textit{$Y$\kern-3pt-\kern1pt supercomplements of $a$}.  If $1\in L$, we call the elements of $(1:a):=(\{1\}:a)$ simply the \textit{supercomplements of $a$}. Observe that $(a,a')\in R^Y$ if and only if $(Y:a) = (Y:a')$.  

\begin{lem} For any subset $Y$ of $L$, $R^Y$ is a $\vee$-congruence on $L$.\end{lem}

\begin{proof} $R^Y$ is clearly an equivalence relation.  Suppose $(a,a')\in R^Y(L)$ and $b\in L$.  Then for all $x\in L$,  $x\vee (b\vee a) = (x\vee b)\vee a \in Y \;\Leftrightarrow (x\vee b)\vee a' = x \vee (b\vee a') \in Y$.  Thus $(a\vee b, a'\vee b)\in R^Y$, as required.
\end{proof}

\begin{defn} A congruence of the form $R^Y$ will be called a \textit{Pierce congruence} (in recognition of \cite{P}).  We write  $R^b$ as shorthand for $R^{\up b}$.  We say \textit{$L$ is conjunctive} if $L$ has $1$ and $R^1(L)$ is equality.\end{defn}

There are several equivalent ways to formulate the conjunctivity condition.  Specifically, the following are clearly equivalent:
\begin{enumerate}
\item $L$ is conjunctive.
\item If two elements of $L$ have the same supercomplements, they are equal.
\item If $a_0$ and $a_1$ are distinct elements of $L$, then there is $w\in L$ such that either \begin{align*}w\vee a_0 &= 1\not=w\vee a_1,\;\text{or}\\ w\vee a_0&\not= 1=w\vee a_1.\end{align*}
\item For all $a, b\in L$ such that $b\not\leq a$, there is $w\in L$ such that  $a\leq w<1 = w\vee b$. 
\item For all $a, b\in L$ such that $a< b$, there is $w\in L$ such that  $a\leq w<1 = w\vee b$. 
\item Every principal ideal of $L$ is an intersection of maximal ideals.  (See Proposition \ref{conjmax}.)
\end{enumerate}

\begin{exmp} A product of conjunctive join-semilattices is conjunctive.  A sub-semilattice of a conjunctive join-semilattice need not be conjunctive. The two-element join-semilattice  $\mathbbm{2}:=\{0,1\}$ is conjunctive, as is $\mathbbm{2}\times \mathbbm{2}$, but the sub-join-semilattice $\{(0,0), (1,0), (1,1)\}\subseteq \mathbbm{2}\times \mathbbm{2}$ is not conjunctive.  \end{exmp} 

\begin{rem}  It follows from the definitions that for any $a\in L$ and any up-set $U\subseteq L$, $(U:a)$ is an up-set of $L$.  The map $a\mapsto (U:a)$ is order-preserving, if we order up-sets by containment.  Thus, $a\mapsto (U:a)$ is an order-isomorphism of $L/R^U$ onto $\{\,(U:a)\mid a\in L\}$.   Note that $(U:a)\cup(U:b)$ is not generally equal to $(U:a\vee b)=((U:a):b) = ((U:b):a)$.  For example, if $L$ is the power set of $\{x,y\}$, $(1:\{x\})\cup(1:\{y\})$ does not contain $\emptyset$, but $(1:\{x,y\}) = (1:1)$ does. \end{rem}

\begin{thm}\label{mainpierce1}Suppose $U\subseteq L$ is an up-set.   Then:
\begin{enumerate}
\item[$(i)$]   $U$ is an $R^U$ class (so $1_{L/R^U} = U$).
\item[$(ii)$]  $L/R^U$ is conjunctive.
\item[$(iii)$]  In any congruence properly stronger than $R^U$, the top class properly contains $U$.
\end{enumerate}
\end{thm}

\begin{proof} 
$(i)$ If $a\in U$, $(U:a) = L$.  If $b\not\in U$, $b\vee b \not\in U$, so $(U:b)\not= L$.

$(ii)$  We use $[a]$ to denote $[a]_{R^U}$.  Suppose $[a]\not=[b]$.  Interchanging $a$ and $b$ if necessary, we may assume that there is $c$ such that  $a\vee c \not\in U$ and $b\vee c\in U$.  Then $[a]\vee\, [c]\not=U$ and $[b]\vee\, [c] =U$.  Thus, $[a]$ and $[b]$ are in different classes of $R^1(L/R^U)$.

$(iii)$  If $R'$ is properly stronger than $R^U$, then $(a, a\vee b) \in R'\setminus R^U$ for some $a, b\in L$.  Therefore, there is $c\in L$ such that $a\vee c\not\in U$ but $a\vee b\vee c\in U$. But $a\vee c\equiv a\vee b\vee c  \mod R'$, so the $R'$ equivalence class containing $U$ also contains an element not in $U$.
\end{proof}

\begin{lem}\label{piercelemma} Suppose $K$ is a conjunctive join-semilattice and $h:L\to K$ is a surjective $\vee$-morphism. Let $V=h^{-1}(1_K)$.  Then $[x]_{R^V}\mapsto h(x):L/R^V\to K$ is an isomorphism.\end{lem}

\begin{proof}  Note that $V:=h^{-1}(1_K)$ is an up-set.  Let $[a]$ denote $[a]_{R^V}$.  For all $a, a'\in L$, the following are equivalent: $(i)$ $[a] = [a']$; $(ii)$ for all $x\in L$, $x\vee a\in V\;\Leftrightarrow\; x\vee a'\in V$; $(iii)$ for all $x\in L$, $h(x)\vee h(a) = 1_K\;\Leftrightarrow\;  h(x)\vee h(a')= 1_K$.  Since $K$ is conjunctive and $h$ is surjective, $(iii)$ is equivalent to $h(a) = h(a')$.  Thus, the map $[a]\mapsto h(a)$ is well-defined and injective.  It is surjective and respects $\vee$ by hypothesis, so it is an isomorphism.\end{proof}

\begin{thm}\label{mainpierce2}Suppose $U\subseteq L$ is an up-set.  Let $Q$ be any congruence on $L$ in which $U$ is a class.  Then, $[x]_Q\mapsto [x]_{R^U}$ is a surjective $\vee$-morphism from $L/Q$ to $L/R^U$.  Thus $R^U$ is the strongest congruence on $L$ in which $U$ is a class. \end{thm}

\begin{proof}  For convenience, set $J:=L/Q$ and let $f:L\to J$ be the canonical surjection.   Let $K:=J/R^1$, and let $g:J\to K$ be the canonical surjection.  Let $h = g\circ f$.  By \ref{mainpierce1}, $(i)$, $h^{-1}(1_K) = U$.   By  Lemma \ref{piercelemma}, $g(f(x))\mapsto [x]_{R^U}:K\to L/R^U$ is an isomorphism, so  $f(x)=[x]_Q\mapsto [x]_{R^U}:J\to L/R^U$ is a surjective $\vee$-morphism.  \end{proof}

The following proposition shows that the results above imply analogous results for distributive lattices, if we strengthen the hypotheses on $U$.

\begin{prop}\label{distpierce}  Suppose $L$ is a distributive lattice, and $U$ is a filter.  Then $R^U$ is a lattice-congruence.\end{prop}

\begin{proof} We need to verify that if $(a, a')\in R^U$, then $(a\wedge b, a'\wedge b)\in R^U$.  This is seen as follows.  For all $x\in L$:
\begin{align*}
x\vee (a\wedge b)\in U \quad &\Leftrightarrow \quad (x\vee a)\wedge (x\vee b)\in U\\
                                         &\Leftrightarrow \quad (x\vee a)\in U\;\&\; (x\vee b)\in U,\quad\text{since $U$ is an up-set,}\\
                                         &\Leftrightarrow \quad (x\vee a')\in U\;\&\; (x\vee b)\in U,
                                         \quad\text{since $(a, a')\in R^U$,}\\
                                         &\Leftrightarrow \quad (x\vee a')\wedge (x\vee b)\in U,\quad\text{since $U$ is a filter,}\\
                                         &\Leftrightarrow \quad x\vee (a'\wedge b)\in U.   \qedhere
\end{align*}
\end{proof}

\begin{prop}  If $c\in L$ is a compact element, then $R^c$ preserves infinite joins.\end{prop}
\begin{proof} Suppose $A_0, A_1\subseteq L$ and $\{\,[a]\mid a\in A_0\,\} = \{\,[a]\mid a\in A_1\,\}$.
Then, for all $x\in L$:
\begin{align*} x\vee \bigvee A_0\geq c &\Leftrightarrow x\vee \bigvee A'_0\geq c,\quad\text{for some finite $A'_0\subseteq A_0$}\\
&\Leftrightarrow x\vee \bigvee A'_1\geq c,\quad\text{for some finite $A'_1\subseteq A_1$}\\
&\Leftrightarrow x\vee \bigvee A_1\geq c.\qedhere
\end{align*}
\end{proof}

%%%%%%%%%%%%%%%%%%%%%%%%
%%%%%%%%%%%%%%%%%%%%%%%%
\subsection{Ideal Conjunctivity}
%%%%%%%%%%%%%%%%%%%%%%%%
%%%%%%%%%%%%%%%%%%%%%%%%

Recall that $L$ refers to an arbitrary join-semilattice.  Since $\ideals L$ is a join-semilattice, all the results of the previous subsection concerning Pierce congruences apply to it.   This must be understood with care.  First of all, the superscript $1$ in $R^1(\ideals L)$ refers to $1_{\ideals L} = L\in \ideals L$.    Second, while $\ideals L$ is complete, the congruence $R^1(\ideals L)$ in general respects only finite suprema.  The canonical map $a\to [a]:\ideals L\to (\ideals L)/R^1$ preserves finite suprema, but it may fail to preserve infinite suprema, since $1_{\ideals L}$ is not compact when $L$ does not have a top element.

\begin{exmp} View $\N$ as a join semilattice with the natural order.  Then $\ideals\N = \{\,\down n\mid n\in \N\,\}\cup\{\N\}$, where $\N = 1_{\ideals\N}$.  For every $n\in \N$, $(\N:\down n)= \{\N\}$.  On the other hand, $(\N:\N)=\ideals\N$.  Thus, $\ideals\N/R^1\cong\mathbbm 2$.    We have 
$\bigvee\{\,\down n\mid n\in \N\,\} = \N$.  However, $[\down n] = [\down 0]$ for all $n\in \N$.  Thus, it is not the case that $\bigvee\{\,[\down n]\mid n\in \N\,\} = [\N]$.  The map $I\mapsto [I]:\ideals\N\to \ideals\N/R^1$ does not preserve infinite suprema.\end{exmp}

\begin{defn}  The restriction of $R^1(\ideals L)$ to $L$ is denoted by   $$R^1(\ideals L)|_L:=R^1(\ideals L)\cap(L\times L).$$  We say that $L$ is \textit{ideally conjunctive} if $R^1(\ideals L)|_L$ is equality. \end{defn} 

\begin{defn}  We say that $W\in\ideals L$ is an \textit{ideal supercomplement of $a\in L$} if  $W\vee a = L$, i.e., the ideal generated by $W\cup \{a\}$ is the improper ideal $L$.\end{defn}

As with the conjunctive property, there are several ways to say that a join-semilattice is ideally conjunctive.    The following are clearly equivalent:
\begin{enumerate}
\item $L$ is ideally conjunctive.
\item   If two elements of $L$ have the same set of ideal supercomplements, they are equal.
\item for any $a_0, a_1\in L$, if $a_0\not= a_1$, there is an ideal $W\in \ideals L$ such that either \begin{align*}W\vee a_0 &= L\not=W\vee a_1,\;\text{or}\\ W\vee a_0&\not= L=W\vee a_1.\end{align*}
\item For all $a, b\in L$ such that $b\not\leq a$, there is $W\in \ideals L$ such that  $a\in W\not =L = W\vee b$. 
\item For all $a, b\in L$ such that $a< b$, there is $W\in \ideals L$ such that  $a\in W\not =L = W\vee b$.
\item Every principal ideal of $L$ is an intersection of maximal proper ideals.  (See Proposition \ref{idealconjmax}.)
\end{enumerate}

\begin{lem} If $L$ has $1$, then $L$ is ideally conjunctive if and only if $L$ is conjunctive.\end{lem}
\begin{proof}  Suppose $a,b\in L$, $b\not\leq a$: $(\Rightarrow)$ By hypothesis, there is a proper ideal $W$ such that $a\in W$ and $W\vee  b =L$.  Therefore, there is $u\in W$ such that $u\vee b = 1$.  Letting $w=u\vee a$, we have $a\leq w<1=w\vee b$.  
$(\Leftarrow)$.   By hypothesis, there is $w\in L$ such that $a\leq w<1=w\vee b$.  Let $W = \down w.$\end{proof}

Proposition 4 of \cite{J84b} is a version of this lemma for distributive lattices. 

\subsection{Problems}
\begin{enumerate}
\item It is natural to ask if $L/(R^1(\ideals L)|_L)$ is always ideally conjunctive.  We cannot mimic the proof of Theorem \ref{mainpierce1}, because $\ideals(L/(R^1(\ideals L)|_L))$ is not the same as $(\ideals L)/R^1$.  We conjecture that $L/(R^1(\ideals L)|_L)$ may fail to be ideally conjunctive.
\item The definition of ideally conjunctive refers to ideals, so it is not a first-order condition in the language of join-semilattices (as is the definition of conjunctivity).  However, this does not preclude the possibility that the property of being ideally conjunctive is first-order.  We conjecture that it is not.
\end{enumerate}

%%%%%%%%%%%%%%%%%%%%%%%%
%%%%%%%%%%%%%%%%%%%%%%%%
\section{Representation Theorems}
%%%%%%%%%%%%%%%%%%%%%%%%
%%%%%%%%%%%%%%%%%%%%%%%%

The purpose of the present section is to show that every conjunctive  join-semilattice is a join-semilattice of open sets forming a subbase for a compact $T_1$ space, and that any join-semilattice map between conjunctive  join-semilattices that satisfies a certain technical condition induces a continuous relation between the representation spaces from which we can recover the map.   Also, we show that every ideally conjunctive  join-semilattice is a join-semilattice of open sets forming a subbase for a $T_1$ space (not necessarily compact).  As an application, we give a generalization of a theorem of Martinez and Zenk.

\begin{exmp}  \textit{A conjunctive join-semilattice may have maximal ideals that are not prime.}   Indeed, the join-semilattice $L$ in Example \ref{maxnotprime} is conjunctive, as we can see by writing out the supercomplements of each element:
\begin{center}
\begin{tabular}{r||c|c|c|c|c}
If $a =$       & $1$ & $xy$              & $xz$             & $yz$              & $\emptyset$\\
\hline
$(1:a)_L =$ & $L$ & $\{1, xz,yz\}$ & $\{1, xy,yz\}$ & $\{1, xy,xz\}$ & $\{1\}$
\end{tabular} \end{center}    \end{exmp}

%%%%%%%%%%%%%%%%%%%%%%%%
\subsection{Representing conjunctive join-semilattices}\label{repthmsec}
%%%%%%%%%%%%%%%%%%%%%%%%

The representation theory for conjunctive join-semilattices rests on the following proposition.

\begin{prop}\label{conjmax}  Suppose $L$ is  a $\vee$-semilattice with $1$. Then $L$ is conjunctive if and only if: for all $a, b\in L$ such that $b\not\leq a$, there is a maximal ideal of $L$ that contains $a$ and does not contain $b$. \end{prop}

\begin{proof} $(\Rightarrow)$  Suppose $a, b\in L$ and $b\not\leq a$.  Select $w\in L$ such that $a\leq w < 1$ and $w\vee b =1$.  There is a maximal ideal $m$ that contains $w$ (and hence $a$).  Because $w\vee b =1$ and $w\in m$, we have $b\not\in m$.   $(\Leftarrow)$  Again, suppose $a, b\in L$ and $b\not\leq a$.  Let $m$ be a maximal ideal such that $a\in m$ and $b\not \in m$.  Then, there is $u\in m$ such that $u\vee b =1$.  Since $a\in m$, $u\vee a<1$.  For $w$, take $u\vee a$.  \end{proof}

\textit{For the remainder of this subsection, we assume $L$ is a conjunctive join-semilattice that contains at least two elements (including $1$).}     Let $\maxx L $ denote the set of all maximal ideals of $L$. For each $a\in L$, let $\widehat a: \maxx L \to \{0,1\}$ be defined by 
$$\widehat a (m): = \begin{cases} 0,& \text{if}\; a\in m;\\1,& \text{if}\; a\notin m.\end{cases}$$
In terms of the map $\phi$ defined after Example \ref{maxnotprime}, we have $\widehat a(m): = \phi_m(a)$.

\begin{lem}\label{replem} The map $$a\mapsto \widehat{a}: L\to \{0,1\}^{\maxx L }$$ is an injective $1$-$\vee$-morphism. 
\end{lem}

\begin{proof}  By Proposition \ref{conjmax}, for any two different elements of  $L$, then there is a maximal ideal of $L$ that contains one and not the other, so the map is injective.  It is clear that $\widehat 1$ is the constant function $1$ on $\maxx L $.  Let $m\in\maxx L $.  Then  $\widehat{a\vee b}\,(m) = 0$ iff $a\vee b \in m$ iff $a\in m\;\&\;b\in m$ iff $\widehat{a}(m) = 0\;\&\;\widehat{b}(m) = 0$  iff $\bigl(\,\widehat{a}\vee \widehat{b}\,\bigr)(m) = 0$.  Thus, $\widehat{a\vee b}=\widehat{a}\vee \widehat{b}$.
\end{proof}

\begin{defn}For each $a\in L$, let $\coz a:=\{\,m\in\maxx L \mid \widehat{a}(m) = 1\,\} = \{\,m\in\maxx L \mid a\not\in m\,\}$.   We call $\coz a$ the \textit{cozero set of $a$}.\end{defn}

\begin{rem} Throughout this paper, we use $\coz a$ to refer to a set of \textit{maximal} ideals.  Observe that $\widehat a$ is the characteristic function of $\coz a$.  The notation $\spec a$, which occurs in Theorem \ref{toprep}, has a similar definition, but it is a set of \textit{prime} ideals.  Since a join semilattice may have maximal ideals that are not prime and prime ideals that are not maximal, in general there is no relationship.  Under the distributive hypothesis (see Section \ref{djsl}), every maximal ideal is prime.    \end{rem}

Let $\mathcal{W}L$ be the weakest topology on $\maxx L $ in which $\coz a$ is open for each $a\in L$.  Let $\SpecM L$ denote $\maxx L$ with the topology $\mathcal{W}L$.

\begin{lem} $\SpecM L$ is $T_1$ and $a\mapsto \coz a$ is an isomorphism of $L$ with a subbase for $\SpecM L$ that is closed under finite joins.\end{lem}   

\begin{proof}   $\mathcal{W}L$ is $T_1$, because given any two maximal ideals, each fails to contain at least one element of the other, so each is in a cozero set not containing the other.  By definition of the topology, the cozero sets form a subbase.  The map $a\mapsto \coz a$ is an injective $\vee$-morphism by the previous lemma.\end{proof}

\begin{rem}  Let $L$ be a subset of $\mathcal P X$.  We say that $L$ is a \textit{$T_1$ subbase} if for any $x,y\in X$, there is $a\in L$ such that $x\in a$ and $y\not\in a$.  It is clearly the case that if $L$ is a $T_1$ subbase, then the topology that it generates is $T_1$.  Conversely, if $L$ is not a $T_1$ subbase, then there are points $x, y\in X$ such that for all $a\in L$, $x\in a\;\Leftrightarrow\;y\in a$.  Then the same thing is true for finite intersections of elements of $L$, and all unions of such sets---hence for the topology generated by $L$. \end{rem}

\begin{defn}For any $B\subseteq L$, we let $\langle B\rangle$ denote the ideal generated by $B$, and let $\llangle B\rrangle$ denote the intersection of all maximal ideals containing $B$. \end{defn}

\begin{exmp}  In general, $\langle B\rangle$ may be a proper subset of $\llangle B\rrangle$.   For example,  suppose $L$ is the usual topology of $[0,1]$, viewed as a join-semilattice.  By the representation theorem (see below), the maximal ideals of $L$ are the points of $[0,1]$.  The collection of all open subsets of $[0,1]$ that omit a neighborhood of $0$ is an ideal that is properly contained in $m_0 =$ the maximal ideal generated by $(0,1]$, but it is not contained in any other maximal ideal.  
\end{exmp}

\begin{lem}  The following are equivalent:
\begin{enumerate}
\item[$(i)$] $\coz a \subseteq \bigcup \{\,\coz b\mid b\in B\,\};$ 
\item[$(ii)$] $a\in \llangle B\rrangle;$
\item[$(iii)$] $\forall x \in L: x\vee a =1 \myimplies \exists\;b\in\langle B\rangle \;\mathrm{such\, that}\;x\vee b = 1.$
\end{enumerate}
\end{lem}

\begin{proof}
$(i\;\Leftrightarrow\; ii)$ Observe that  
$$\bigcup\{\,\coz b\mid b\in B\,\} = \{\,m\in\maxx L \mid B\not\subseteq m\,\}= \{\,m\in\maxx L \mid \llangle B\rrangle \not\subseteq m\,\}.$$  Thus, $(i)\iff \bigl(\forall m\in \maxx L,\;a\not\in m \myimplies \llangle B\rrangle \not\subseteq m\bigr)\iff (ii)$.   

$(iii\myimplies ii)$   Let $m$ be a maximal ideal containing $B$.  Toward a contradiction, suppose $a\not\in m$.  Then $x\vee a = 1$ for some $x\in m$.  Assuming $(iii)$, it follows that there is $b\in\langle B\rangle$ such that $x\vee b = 1$.  But this is impossible, since $x$ and $b$ are both in $m$.  Hence, $a\in m$.

$(ii\myimplies iii)$ Suppose $a$ does not satisfy $(iii)$.   Then there is $c\in L$ such that $c\vee a =1$, while $B\cup\{c\}$ is contained in a proper ideal, and hence is contained in a maximal ideal $m_0$.  Clearly, $a\not\in m_0$, so $a\not\in  \llangle B\rrangle$.
\end{proof}

\begin{lem} $\SpecM L$ is compact.\end{lem}
\begin{proof} Setting $a=1_L$  in the previous lemma and letting $x$ be any element of $B$  we see that $\coz 1_L \subseteq \bigcup \{\,\coz b\mid b\in B\,\}$ if and only if $1\in\langle B\rangle$.  Thus any cover of $\SpecM L(=\coz 1_L )$ by elements of the subbase $\{\,\coz b\mid b\in L\,\}$ has a finite subcover.  Compactness of $\SpecM L$ then follows from the Alexander subbase theorem.   \end{proof}

\begin{lem}\label{replemma} Suppose $X$ is a set with at least two elements and $L\subseteq \mathcal P X$ is a $T_1$ subbase for a compact topology on $X$.  Suppose further that $L$ contains $X$ and is closed under joins.  Then $L$ is conjunctive, and the map $x\mapsto m_x:=\{\,a\in L\mid x\not\in a\,\}$ is a homeomorphism of $X$ with $\SpecM L$.\end{lem}

\begin{rem}The conclusions of the theorem are also true if $X$ contains a single point and $L=\mathcal P X$.\end{rem} 

\begin{proof}  First, we show that for each $x\in X$, $m_x$ is a maximal ideal of $L$.  By the $T_1$ hypothesis, $m_x$ is not empty.   Suppose  $a\in L\setminus m_x$ (so $x\in a$). Since $L$ is a $T_1$ subbase, for every $y\in X\setminus \{x\}$, $m_x$ contains an open neighborhood $b_y$ of $y$.    Because $X$ is compact, there is a finite set $Y\subset X\setminus\{x\}$ such that $X\subseteq a\vee b$, where $b=\bigvee\{\,b_y\mid y\in Y\,\}\in m_x$.   Second, we show that $L$ is conjunctive.   Suppose $a,b\in L$ and  $b\not\subseteq a$.  Pick $x\in b\setminus a$.  Then $m_x$ contains $a$ and does not contain $b$.  By Proposition \ref{conjmax}, $L$ is conjunctive.  Third, we show that the map $x\mapsto m_x: X\to \SpecM L$ is bijective.  It is injective by the $T_1$ subbase hypothesis.  It is surjective, for suppose $m\in \SpecM L$.  Since $m$ is an ideal, $\{\, X\setminus a\mid a\in m\,\}$ is a family of closed subsets of $X$ with the finite intersection property, and since $X$ is compact, there is at least one point that all these sets have in common.  There can be no more than one, since $m$ is maximal.  Finally, $x\mapsto m_x$ is a homeomorphism: by definition, $\coz a =  \{\,m_x\mid x\not\in a\}$, so $a\mapsto \coz a$ is a bijection  between subbases for the topologies on $X$ and $\SpecM L$.\end{proof}

The following theorem summarizes all the lemmas in this subsection.

\begin{thm}[Representation Theorem for Conjunctive Join-Semilattices]\label{repthm}$\phantom{x}$

$(i)$ Suppose $L$ is a conjunctive join-semilattice with at least two elements.  Let $\maxx L$ be the set of maximal ideals of $L$ and let \begin{align*}\coz: L&\to \mathcal{P}(\maxx L)\\ a&\mapsto \coz a:=\{\,m\in \maxx L\mid a\notin m\,\}.\end{align*}  Then $\coz$ is a join-semilattice injection and its image $\coz L$ is a subbase for a compact $T_1$ topology (which we call $\mathcal{W}L$) on $\maxx L$.

$(ii)$ Suppose $X$ is a non-empty set and $\mathcal{T}$ is a compact $T_1$ topology on $X$.  Further, suppose that $L$ is a subbase for $\mathcal{T}$ that is closed under finite unions.  Then $L$ is conjunctive and $x\mapsto m_x:=\{\,a\in L\mid x\notin a\,\}$ is a homeomorphism of $(X, \mathcal{T})$ with $(\maxx L, \mathcal{W}L)$. \qed   \end{thm}

\subsection{Functoriality} 

We begin by summarizing the functorial nature of Stone's representation for distributive lattices.  
Let $A$ be a bounded distributive lattice.  $\Spec A$ denotes the set of prime ideals of $A$.   For each $a\in A$, define $\widehat a:\Spec A\to \{0,1\}$ by $$\widehat{a}(p) = \begin{cases}0&\text{if $a\in p$}\\1&\text{otherwise.}\end{cases}$$  Suppose $\phi:A\to B$ is a morphism of bounded distributive lattices.    
For $q\in \Spec B$, define $f_\phi(q):=\phi^{-1}(q) \in \Spec A$. 
Then $f_\phi$ is a function from $\Spec B$ to $\Spec A$.  Moreover  $\widehat a\circ f_\phi = \widehat {\phi(a)}$.    The usual topology on $\Spec A$ is the weakest topology in which $\widehat a^{-1}(1)$ is open for all $a\in A$ (and similarly for $B$).  With respect to these topologies, $f_\phi$  is continuous, for $f_\phi^{-1}(\widehat a^{-1}(1)) =(\widehat a\circ f_\phi)^{-1}(1) = \widehat {\phi(a)}^{-1} (1)$, so the inverse image of any basic open is open. 

If we attempt to replicate this in the category of conjunctive join semilattices (using maximal ideals rather than prime ideals), some modifications are necessary.  First, there is no hope of representing every morphism of conjunctive join-semilattices for the following reason.  We have mentioned above that for any ideal $I$ in a join-semilattice $L$, the map $\phi_I:L\to \{0,1\}$ defined by $\phi_I(a) = 0$ iff $a\in I$ is a $1$-$\vee$-morphism, and we have given examples that illustrate that there is a conjunctive join-semilattice $L$ with an ideal $I\subseteq L$ that is \textit{not} an intersection of maximal ideals of $L$. For such an ideal, there is evidently no way to determine $\phi_I(a)$ for each $a\in L$ from $\widehat a$, since the only information we can extract from $\widehat a$ is the set of maximal ideals to which $a$ belongs.  We address this problem by excluding such deviant morphisms from consideration.  We will seek representations only for morphisms that satisfy the following definition. 

\begin{defn} \label{conjmorph} Suppose $\phi:L\to M$ is a $1$-$\vee$-morphism of join-semilattices.  We say that $\phi$ is a \textit{conjunctive morphism} if $\phi^{-1}(w)$ is an intersection of maximal ideals of $L$ whenever $w$ is a maximal ideal of $M$.\end{defn}

The second problem is that even when $\phi$ is conjunctive, taking inverse images does not yield a function from $\SpecM M$ to $\SpecM L$.  We must deal with the fact that for $w\in \SpecM M$, $\phi^{-1}(w)$ is in general only an intersection of maximal ideals.  We address this problem by replacing $f_\phi$ with a multi-valued function $Q_\phi$.  Using a relation $Q_\phi$ in place of a function $f_\phi$ raises a third problem.  How do we compose $\widehat a$ with a relation?  The idea is to use the join in the complete semilattice $\mathbbm 2$, as we show in the next paragraph. 

To simplify notation, we use $X_L$ as shorthand for $\SpecM L$.  Let $X_L$ and $X_M$ be the representation spaces for conjunctive join-semilattices $L$ and $M$.  
Suppose $a \in L$ and $\widehat a :X_L \to \mathbbm{2} = \{0,1\}$ is its representation.  
Let $Q\subset X_M\times X_L$ be any relation such that for all $w\in X_M$, there exists at least one $v\in X_L$ such that $(w,v)\in Q$.  Thus, $Q:w\mapsto v$ is a ``multi-valued function'' from $X_M$ to $X_L$.  
For any $a\in L$, define $\bigvee\,\widehat{a}\circ Q: X_M\to \mathbbm 2$ 
by:
$$\bigvee\, \widehat a \circ Q(w):= \bigvee\{ \widehat a(v)\mid v\in X_L\;\&\;(w,v)\in Q\,\}.$$  
Note that 
$$\bigvee\, \widehat{1_L} \circ Q  = \widehat{1_M}.$$
Now suppose $b \in L$ and $\widehat b :X_L \to \mathbbm 2$ is its representation.  For any $w\in X_M$,
\begin{align*}
\bigvee\, \widehat{a\vee b} \circ Q(w)
     &= \bigvee\bigl\{\,\widehat{a\vee b}\,(v)\bigm| v\in X_L\;\&\;(w,v)\in Q\,\bigr\}\\
     &= \bigvee\bigl\{\,\widehat{a}(v)\vee \widehat{b}(v)\bigm| v\in X_L\;\&\;(w,v)\in Q\,\bigr\}\\
     &= \bigvee\bigl\{\,\widehat{a}(v)\bigm| v\in X_L\;\&\;(w,v)\in Q\,\bigr\}\,\vee\,\bigvee\bigl\{\, \widehat{b}(v)\bigm| v\in X_L\;\&\;(w,v)\in Q\,\bigr\}\\
     &= \bigl(\,\bigvee\, \widehat{a}\circ Q\,\bigr)(w)\vee \bigl(\,\bigvee\, \widehat b\circ Q\,\bigr)(w)
\end{align*}
Thus, we see that $a\mapsto \bigvee\, \widehat a\circ Q$ is a $1$-$\vee$-morphism from $L$ to $\mathbbm 2^{X_M}$.  Whether or not $\bigvee\, \widehat a\circ Q$ lies in $\widehat M$ depends, of course, on $Q$.

We now show that for any conjunctive morphism $\phi:L\to M$ between conjunctive join-semilattices, we have a natural relation $Q_\phi$ that induces $\phi$.  Let
$$Q_\phi:=\{\,(w,v)\in X_M\times X_L\mid \phi^{-1}(w)\subseteq v\,\}.$$

\begin{prop} If $\phi:L\to M$ is a conjunctive morphism between conjunctive join-semilattices, then $\bigvee\, \widehat{a} \circ Q_\phi=\widehat{\phi(a)}$ for all $a\in L$.   \end{prop}

\begin{proof} 
For any maximal ideal $w$ of $M$,
$$\bigvee\,\widehat{a} \circ Q_\phi(w)
     = \bigvee\{ \widehat{a}(v)\mid v\in X_L\;\&\;\phi^{-1}(w)\subseteq v\,\},$$
by the definitions of $\bigvee\,\widehat{a} \circ Q$ and $Q_\phi$.  Thus, for any $w$,
\begin{align*}
\bigvee\,\widehat{a} \circ Q_\phi(w) =0 
     &\iff\; \text{for all $v\supseteq \phi^{-1}(w)$, $\widehat{a}(v) = 0$}\\
     &\iff\; \text{for all $v\supseteq \phi^{-1}(w)$, $a\in v$}\\
     &\iff\; a\in \phi^{-1}(w)&&\text{since $\phi$ is conjunctive}\\
     &\iff\; \phi(a)\in w\\
     &\iff\;\widehat{\phi(a)}(w)=0.&&\qedhere
\end{align*}
\end{proof}

It is interesting to ask if $Q_\phi^{-1}(V):=\{\,w\in X_M\mid \exists v\in V\;\phi^{-1}(w)\subseteq v \,\}$ is an open subset of $\SpecM M$, whenever $V$ is an open subset of $\SpecM L$.  Note first that
\begin{align*}
w\in Q_\phi^{-1} (\coz a) 
\iff &\exists v\in X_L\;\text{such that}\; \phi^{-1}(w)\subseteq v\;\&\;a\not\in  v\\
\iff  & a\not\in \phi^{-1}(w)\\
\iff & \phi(a)\not\in w\\
\iff & w \in \coz \phi(a).
\end{align*}
From this, we see that for any $a\in L$, $Q^{-1}(\coz a) = \coz\phi(a)$.  Second, note that for any collection $\{A_j\}_{j\in J}$ of subsets of $X_L$, and any relation $Q:X_M\to X_L$, $Q^{-1}(\bigcup_{j\in J}A_j) = \bigcup_{j\in J}Q^{-1}(A_j)$.  Intersections of cozero sets, however, are problematic.  If $a,b\in L:
$\begin{align*}
w\in Q_\phi^{-1} (\coz a\cap \coz b) 
\iff &\exists v\in X_L\;\text{such that}\; \phi^{-1}(w)\subseteq v\;\&\;a\not\in  v\;\&\;b\not\in  v\\
\Longrightarrow\; & a\not\in \phi^{-1}(w)\;\&\; b\not\in \phi^{-1}(w)\\
\iff & \phi(a)\not\in w\;\&\; \phi(b)\not\in w\\
\iff & w \in \coz \phi(a)\cap\coz \phi(b).
\end{align*}
This shows that $Q^{-1}(\coz a\cap\coz b) \subseteq  \coz\phi(a)\cap \coz\phi(b)$, but we see no reason to expect equality, since the complement of $\phi^{-1}(w)$ need not be a filter.  We do not know if  $Q^{-1}(\coz a\cap\coz b)$ is open.  Perhaps there is some other form of continuity for multifunctions that $Q$ satisfies.

\begin{exmp}  Let $L$ be the distibutive lattice of all cofinite subsets of $\N$, together with $0_L = \emptyset$, so $X_L = \{\,m_x\mid x\in \N\,\}$, where $m_x$ is the maximal ideal consisting of all elements of $L$ that do not contain $x$.  Note that $\{0_L\}$ is a prime ideal of $L$ that is not maximal, but it is the intersection of all the maximal ideals.  Let $M = \{0_M,1_M\}$, so $X_M$ is the one-point space, $\{\ast\}$, where $\ast:=\{0_M\}$. Let $\phi: L\to M$ satisfy $\phi(a) = 0$ iff $a=\emptyset$.     Then, $\phi$ is a lattice morphism, and evidently $Q_\phi = X_M\times X_L$.  If $a\in L$, $a \not= 0_L$ then $Q^{-1}(\coz a) = \{\ast\}$.  On the other hand, $\coz 0_L =\emptyset$ and $Q^{-1}(\coz 0_L) = \emptyset = \coz 0_M$.\end{exmp}

\subsection {Representing Ideally Conjunctive Join-semilattices}

We prove an analog of Proposition \ref{conjmax} for ideally conjunctive join-semilattices.

\begin{prop}\label{idealconjmax}  Suppose $L$ is  a join-semilattice. Then $L$ is ideally conjunctive if and only if: for all $a, b\in L$ such that $b\not\leq a$, there is a maximal proper ideal of $L$ that contains $a$ and does not contain $b$. \end{prop}

\begin{proof} $(\Rightarrow)$  Suppose $a, b\in L$ and $b\not\leq a$.  Select $W\in \ideals L$ such that $W\vee b =1$ and $a\in W \not= L$.  Clearly, $b\not\in W$, so there is an ideal $W'$ containing $W$ and maximal missing $b$.  Such a $W'$ is maximal proper, because any ideal that properly contains $W'$ contains both $W$ and $b$ and hence is equal to $L$. $(\Leftarrow)$  Suppose $a$ and $b$ are different elements of $L$.  Without loss of generality, $b\not\leq a$.  Let $m$ be a maximal ideal of $L$ that contains $a$ and does not contain $b$.  Then $m$ is a supercomplement of $b$ but not of $a$, so  $a$ and $b$ have different sets of supercomplements.
\end{proof}

The following corollary is analogous to part $(i)$ of the Representation Theorem for Conjunctive Join-Semilattices  (Theorem \ref{repthm}).  This corollary is weaker in that it does not assert that $(\maxx L, \mathcal{W}L)$ is compact.  The proof parallels that of Theorem \ref{repthm}.

\begin{cor}\label{dlrep} Suppose $L$ is an ideally conjunctive join-semilattice with at least two elements.  Let $\maxx L$ be the set of maximal proper ideals of $L$ and let \begin{align*}\coz: L&\to \mathcal{P}(\maxx L)\\ a&\mapsto \coz a:=\{\,m\in \maxx L\mid a\notin m\,\}.\end{align*}  Then $\coz$ is a join-semilattice injection and its image $\coz L$ is a subbase for a $T_1$ topology $\mathcal{W}L$ on $\maxx L$.\qed\end{cor} 

\begin{rem} If $L$ is distributive, we can assert that $\coz L$ is a base.  This follows from Proposition \ref{toprep}, below, since when $L$ is distributive, $\SpecM L$ is a subspace of $\Spec L$.  \end{rem}

\subsection{Application: Yosida Frames}\label{mzsect}

We close this section by making a connection to work of Martinez and Zenk.    A complete lattice that is isomorphic to $\ideals L$ for some distributive join-semilattice $L$ is called an \textit{algebraic frame}.  (We discuss join-semilattices that satisfy the distributivity condition in detail in Section \ref{djsl}, below.)   In \cite{MZ}, Martinez and Zenk define a \textit{Yosida frame} to be an algebraic frame in which every compact element is the infimum of the maximal elements above it. They define a frame $F$ to be \textit{finitely subfit} if: for all compact elements $b\not\leq a$ in $F$, there is some  $w\in F$ such that $a\leq w < 1$ and $w\vee b =1$.  One of the main results of \cite{MZ} is their Proposition 4.2:  \textit{Suppose $A$ is an algebraic frame with $FIP$ (i.e., the meet of any two compact elements is compact).  Then $A$ is Yosida if and only if it is finitely subfit}.   To translate this into our terminology, $A$ is a finitely subfit algebraic frame if and only if $A$ is isomorphic to the lattice of ideals of some ideally conjunctive distributive join-semilattice.   With this observation, we can see that a much stronger statement than  \cite{MZ} Proposition 4.2 is an immediate corollary of our Proposition \ref{idealconjmax}.     

\begin{cor}\label{mzprop} Let $L$ be a join-semilattice. Every principal ideal of $L$ (i.e., every compact element of $\ideals L$) is the meet of the maximal ideals that contain it if and only if $L$ is ideally conjunctive.\qed
\end{cor}

Observe that the ideals of a lattice are defined entirely in terms of the join operation.  In other words, if $\mathbf{F}$ is the forgetful functor from lattices to join-semilattices, and $L$ is a lattice, then $\ideals L = \ideals  \mathbf{F}L$. Therefore, the corollary applies to all algebraic frames, and of course much more. The Martinez-Zenk result is an immediate consequence.  The corollary shows that the FIP hypothesis is unnecessary.  In fact, there is no need to refer to meets, except as a way to restate the condition in Proposition \ref{idealconjmax}.  Nor are any hypotheses about distributivity needed.  

\subsection{Questions}
\begin{enumerate}
%\item The representation is not functorial.  Suppose $L$ is conjunctive and $\phi:L\to \{0,1\}$ is a function that preserves $\vee$ and $1$.  Then $\phi^{-1}(0)$ need not be maximal.  For example, let $L$ be the cofinite topology on an infinite set.  If we define $\phi(a) = 1$ unless $a=\emptyset$, then  $\phi^{-1}(0) =\{\emptyset\}$, which is not maximal.  Are there reasonable hypotheses on $L$ under which an ideal of $L$ is maximal if and only if it is $\phi^{-1}(0)$ for some $\phi:L\to \{0,1\}$?

\item Let $wL$ be the sublattice of $\mathcal{W}L$ generated by $\{\,\coz a \mid a\in L\,\}$.  As a sublattice of the power set of $\maxx L$, $wL$ is distributive.  Can we characterize the relationship between $L$ and $wL$ algebraically?  What is the relationship between $\SpecM L$ and $\Spec wL$?  (We answer these questions in subsection \ref{join-generation}.)

\item Since the criterion for $\{\,\coz b\mid b\in B\,\}$ to cover $\coz a$ is not finite,  $\coz a$ in general is not compact.  For example, let $L$ be the join semilattice of all open subsets $[0,1]$.  By the Representation Theorem, part $(ii)$,  $\SpecM L$ is homeomorphic with $[0,1]$ and the cozero sets are open subsets of $[0,1]$.    Under what conditions is $\coz a$ compact?  Characterize those $L$ for which $\coz a$ is compact for all $a\in L$.

\item The conjunction condition for join-semilattices is first-order in the language of join-semilattices, but though it is equivalent to the higher-order condition that every principal ideal is an intersection of maximal ideals.  Is there a first-order way of stating the conjunction condition for morphisms (Definition \ref{conjmorph})? (This must avoid direct reference to maximal ideals.)
\end{enumerate}

%%%%%%%%%%%%%%%%%
%%%%%%%%%%%%%%%%%
\section{Distributive join semilattices}\label{djsl}
%%%%%%%%%%%%%%%%%
%%%%%%%%%%%%%%%%%

A join-semilattice $L$ is said to be \textit{distributive} if $$\text{$\forall\, a, b_0, b_1\in L$: $a\leq b_0\vee b_1\myimplies\exists\, a_0, a_1\in S$ s.t.  $a_0\leq b_0\;\&\; a_1\leq b_1\;\&\; a=a_0\vee a_1$.}$$   The distributive join semilattices are important for two reasons.  First, these are the most general semilattices for which there is a good topological representation theory; see \cite{G}.  Second, the compact elements of an algebraic frame form a distributive join semilattice with $0$ and every  distributive join semilattice arises this way.   Distributivity proves to be a powerful but subtle property.

\subsection{Algebraic Frames and Prime Spectra}
In the present subsection, we review results from \cite{G} concerning the relationship of a distributive join-semilattice $L$ to its frame $F=\ideals L$ of join-semilattice ideals.  In a few cases, we supply details that are relevant to present work and are not fully elaborated in \cite{G}.  We provide specific references to the relevant content of \cite{G}.

Recall that a \textit{frame} $F$ is a complete lattice with $0$ and $1$ in which the binary meet operation $\wedge$ distributes over any join:
$$a\wedge \bigvee B = \bigvee \{\,a\wedge b\mid b\in B\,\}.$$  A \textit{frame morphism} is a map between frames that preserves $\wedge$ and  $\bigvee$.  An element $c\in F$ is \textit{compact} if for all subsets $B$ of $F$, $c\leq \bigvee B$ implies $c\leq \bigvee B'$ for some finite $B'\subseteq B$.  The set of compact elements of $F$ is a $\vee$-semilattice and is denoted $\cpt F$.  A frame is said to be \textit{algebraic} if every element is a join of compact elements.
   
\begin{lem}\label{djl1}  Let $L$ be a join semilattice, and let $\ideals L$ denote the complete lattice of ideals of $L$.  The following are equivalent:
\begin{enumerate}
\item[$(i)$]  $L$ is distributive.
\item[$(ii)$] $\ideals L$ is a distributive lattice.
\item[$(iii)$] $\ideals L$ is a frame, i.e., it satisfies the infinite distributive law:
 $$\text{
 for all $I\in \ideals L$ and $\mathcal J\subseteq \ideals L$, $I \wedge\bigvee \mathcal J =  \bigvee\{\,I\wedge J\mid J\in \mathcal J\,\}$}.$$ 
\end{enumerate}
\end{lem} 

\begin{proof} $(i\Leftrightarrow ii)$ is Lemma 184 of \cite{G}.  $(iii\Rightarrow ii)$ is obvious.  $(i\Rightarrow iii)$.  Let $\mathcal J$ be a set of ideals.  The distributivity of $L$ implies that if $a\leq b_1\vee\cdots\vee b_n$, then 
 $a=a_1\vee\cdots\vee a_n$ for some $a_i\leq b_i$.  Accordingly,  each element of $\bigvee \mathcal J$ is of the form $\bigvee A$, where $A$ is a finite subset of $\bigcup \mathcal J$.  Using this, we show the infinite distributive law.  $I\cap\bigvee \mathcal J$ consists of the  $\bigvee A$ that are in $I$, but if $\bigvee A\in I$, then each $a\in A$ belongs to $I\cap J$ for some $J\in \mathcal J$.  Thus, $I\cap\bigvee \mathcal{J}\subseteq \bigvee\{\,I\wedge J\mid J\in \mathcal J\,\}$.  The other containment is obvious. \end{proof}
 
 Suppose $L$ is a distributive join semilattice and $\ideals L$ is its frame of ideals.  Clearly, every $J\in \ideals L$  is a supremum of principal ideals.  Moreover, $J\in \ideals L$ is compact if and only if it is principal.  Thus, $\ideals L$ is algebraic.
 
\begin{prop}  If a frame $F$ is algebraic, then $\cpt F$ is a distributive join semilattice, and $F\cong \ideals\cpt F$.\end{prop} 

\begin{rem}The first assertion is a point-free version of (part of) Theorem 191 of \cite{G}.  For the convenience of the reader, we include the proof.   Note that the $\vee$ operation of $F$ always restricts to a $\vee$ operation on $\cpt F$.  The $\wedge$ operation of $F$ does not---$\cpt F$ is not in general closed under meets.  \end{rem}

\begin{proof}  Given $a, b_0, b_1\in \cpt F$ with $a\leq b_0\vee b_1$, we must find $a_0, a_1\in \cpt F$ with $a_0\leq b_0$ and $a_1\leq b_1$, and $a=a_0\vee a_1$.  Since $F$ is join-generated by $\cpt F$, we have: $a\wedge b_i = \bigvee A_i$ for some $A_i\subseteq  \cpt F$, $i=0,1$.  Using the operations in $F$, $a = (a\wedge b_0)\vee (a\wedge b_1) = \bigvee A_0\vee \bigvee A_1$.  Since $a$ is compact, there are finite sets $A'_0\subseteq  A_0$ and $A'_1\subseteq  A_1$ such that
$a  =  \bigvee A'_0\vee \bigvee A'_1$.  The elements $a_i:= \bigvee A'_i$, $i=0,1$, are in $\cpt F$ and satisfy the required conditions.  The second assertion follows from the observation that the following maps are inverses of one another:
\begin{align*}
\ideals \cpt F\ni I&\mapsto \bigvee I\in F, \text{and}\\
F\ni f&\mapsto(\down f\cap \cpt F)\in \ideals \cpt F.\qedhere
\end{align*}
\end{proof}

The following lemma and proposition give the main features of the topological representation theorem for distributive join-semilattices.

\begin{lem}\label{djl2} Suppose $L$ is a distributive join semilattice.  Let $F\subseteq L$ be a filter, and suppose $I$ is an ideal maximal disjoint from $F$.  Then $I$ is prime (i.e., its complement is a filter).\end{lem}  

\begin{proof}Suppose $a, b\not\in I$.  
We must show that $a$ and $b$ have a lower bound that is not in $I$.  
Since $I$ is maximal disjoint from $F$,  $a\vee p\in F$ and $b\vee q\in F$ for some $p,q\in I$.   
Since $F$ is a filter, there is $c\in F$ such that $c\leq a\vee p$ and $c\leq b\vee q$.  
By distributivity, $c=a'\vee p'$, with $a'\leq a$ and $p'\leq p$.  
But $a'\leq c\leq b\vee q$, so by distributivity again $a' = b' \vee q'$, with $b'\leq b$ and $q'\leq q$. Now $b'$ is a lower bound for both $a$ and $b$.   
Moreover, $b'\notin I$, since $b'\vee q'\vee p'=c\not\in I$, while $q'\vee p'\in I$.\end{proof}

Note that in the proof above, we may take $p=q$, since $p\vee q\in I$.  Let us examine how the proof runs in the special case when $F = \{1\}$.  Suppose $I$ is maximal and $a,b\not\in I$.  Then  $a\vee p=1=b\vee p$ for some $p\in I$.  We apply distributivity once: since $a\leq b\vee p$, we have $a = b'\vee p'$ with $b'\leq b$ and $p'\leq p$.  So, $b'\leq a$ and $b'\leq b$.  Moreover, $1=a\vee p = b'\vee p'\vee p = b'\vee p\not\in I$.  Since $p\in I$, $b'\not\in I$.

\begin{prop}\label{toprep}  Suppose $L$ is a distributive join semilattice.  Let $\Spec L$ denote the set of prime ideals of $L$ with the topology generated by sets $$\spec a:= \{\,p\in \Spec L\mid a\not\in p\,\},$$ where $a\in L$.  Then the map $J\mapsto \{\,p\in \Spec L\mid J\not\subseteq  p\,\}$ is a frame-isomorphism of $\ideals L$ with the topology of $\Spec L$.\end{prop} 

\begin{proof} This is a restatement of \cite{G}, Lemma 186.\end{proof}

\bigskip

\subsection{Distributivity and maximal vs.~prime ideals in conjunctive join-semilattices}\label{maxanddist}

We have shown that in a distributive join semilattice, all maximal ideals are prime.  Suppose $L$ is a conjunctive join semilattice that is not distributive. Does $L$ have a maximal ideal that is not prime?  In this subsection, we answer in the affirmative for complete (and in particular, for finite) join semilattices.  

Let $L$ be a finite conjunctive join semilattice.  We can represent $L$ as a $T_1$ subbase on a finite set $X$.  $\SpecM L$ is identified with $X$ in the discrete topology.  Each point $x\in X$ is identified with the maximal ideal $m_x:=\{\,a\in L\mid x\not\in a\,\}$.   

Suppose $L$ fails to be distributive.  Then there are $a,b,c\in L$ such that: $a\subseteq b\cup c$ and 
$$\text{$\forall\;b', c'\in L,\;\;b'\subseteq  a\cap b\;\&\;c'\subseteq  a\cap c\;\;\Rightarrow\;\; b'\cup c'\neq a$.}$$ (Note that $a\cap b$ and $a\cap c$ need not be in $L$.)  Since $X$ is finite, there is a largest $b' \in L$ such that $b'\subseteq a\cap b$ and a largest $c'\in L$ such that $c'\subseteq a\cap c$.   

\begin{minipage}{.6\textwidth} 
Now, $a\neq b'\cup c'$, so $b'\cup c'$ does not contain \textit{either}:
	\begin{enumerate}
		\item[($i$)] some $x\in a\setminus c$,   \textit{or} 
		\item[($ii$)]  some $y\in a\cap b\cap c$,  \textit{or}
		\item[($iii$)]  some $z\in a\setminus b$. 
	\end{enumerate}
\end{minipage} \hskip 4em
\begin{minipage}{.4\textwidth}
	\includegraphics[width=0.35\textwidth]{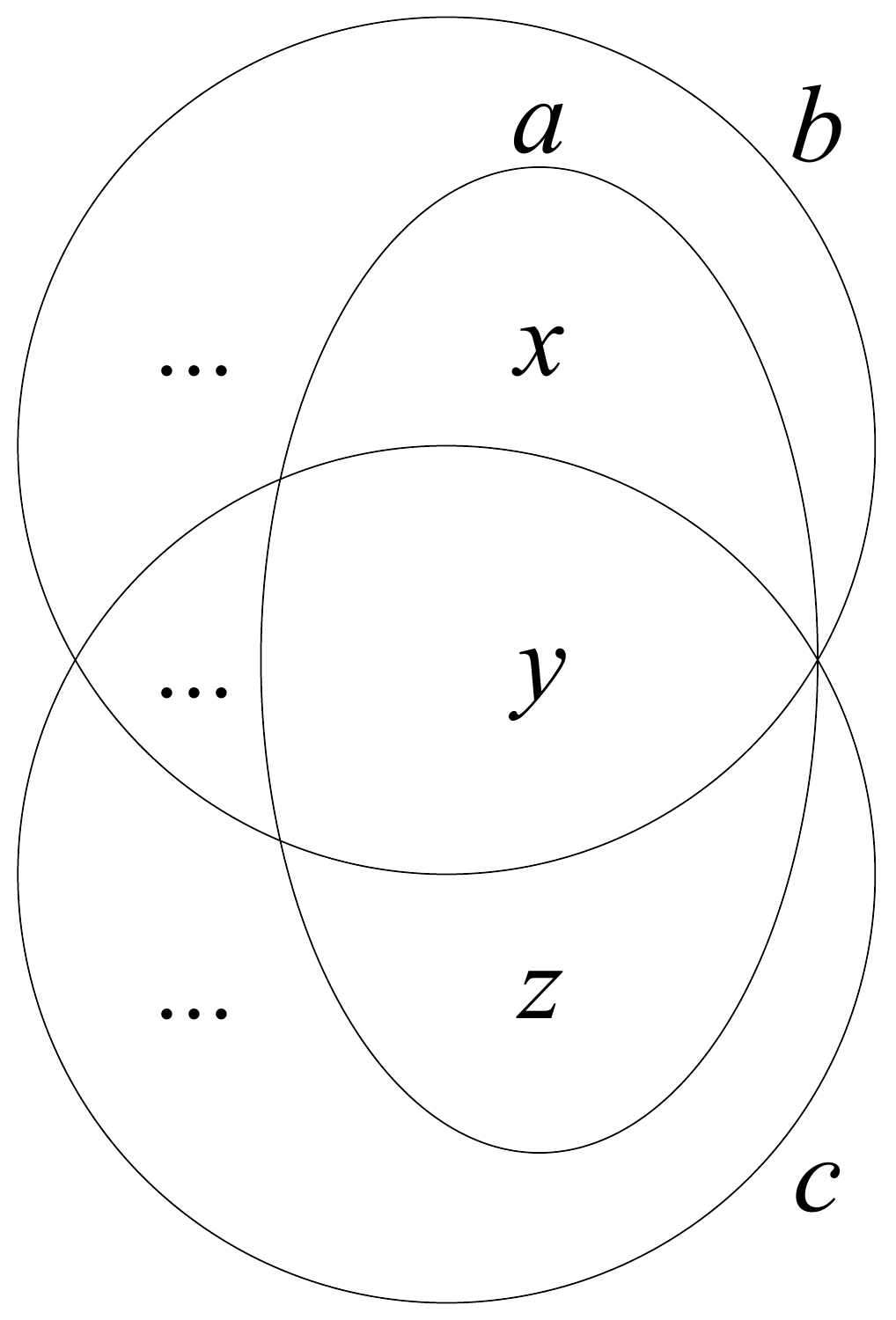}
\end{minipage}

\smallskip
\noindent In case ($i$), $a$ and $b$ contain $x$, i.e., are not in $m_x$, but any lower bound in $L$ for $a$ and $b$ does not contain $x$, and therefore is in $m_x$.  Case ($iii$) is similar: $a$ and $c$ are not in $m_z$, but any lower bound in $L$ for $a$ and $c$ is in $m_z$.
In case ($ii$), $a$, $b$ and $c$ all fail to belong to $m_y$, but the only lower bounds for $a$ and $b$  in $L$ are in $m_y$, and similarly for $a$ and $c$.  Thus, we have shown that \textit{any finite conjunctive join semilattice that fails to be distributive has a maximal ideal that is not prime.}

The argument used above in the finite case depends only on the fact that every subset of $L$ has a \textit{least} upper bound.  Thus, we may replace ``finite'' with ``complete'' in the assertion proved above.   

\begin{prop} Any complete conjunctive join semilattice that fails to be distributive has a maximal ideal that is not  prime.\qed\end{prop}

\subsection{Questions.}
\begin{enumerate}   
\item Rhodes \cite{R} mentions that an epimorph of a distributive join semilattice need not be distributive.  State useful conditions on $L$ and an up-set $U$ (possibly $\{1\}$) that are equivalent to $L/R^U$ being distributive.  
\item Is it true in general (i.e., without the completeness assumption) that a conjunctive join semilattice that fails to be distributive has a maximal ideal that is not  prime?
\item Examine the relationship between the representation theory for distributive join-semilattices and the representation theory for conjunctive join-semilattices when both hypotheses are satisfied.
\end{enumerate}

\section{Distributive Lattices}

\subsection{Relationships to the Wallman compactification}

Let $X$ be a topological space.  We call a base $\B$ for the topology of $X$ \textit{annular} if  $\emptyset$ and $X$ are in $\B$ and $\B$ a sublattice of the frame of open sets of $X$.   A \textit{Wallman base for $X$} is an annular base $\B$ such that:
$$\text{If $u\in U \in \B$, then there exists $V\in \B$ with $u\not\in V$ and $U \vee V= X$.}$$

\begin{lem}  Let $\B$ be an annular base for $X$.  Then $X$ is a Wallman base if and only if, for each $u\in X$, the ideal $m_u:=\{\,V\in \B\mid u\not\in V\,\}$ is maximal. \qed \end{lem} 

A $T_0$ space has a Wallman base if and only if it satisfies the $T_1$ separation axiom; see \cite{GM}. We generally require $T_0$.  In any case, we will always be explicit about separation assumptions.  

If $\B$ is an annular base containing $\{x\}^c$ (the complement of the point $x$) for each $x\in X$, then evidently $X$ is $T_1$ (points are closed) and $\B$ is Wallman. 

\begin{lem} Let $\B$ be an annular base for a compact $T_1$ space $X$.  Then $\B$ is Wallman. \end{lem}

\begin{proof} Suppose $u\in U\in \B\setminus m_u$.   We can cover $X$ with $U$ together with neighborhoods $V_x$ of each $x\not\in U$, with each $V_x$ not containing  $u$.  Finitely many $V_x$ will do, and their union, call it $V$, also misses $u$.  Moreover, $U\cup V = X$.\end{proof}

Recall that $\B$ is conjunctive if for any  $W$ and $U$ in $\B$ with $W\subsetneq U$, there is $V\in \B$ such that $V\cup W\subsetneq X=  V \cup U$.

\begin{lem} Let $X$ be a topological space.  Any Wallman base for $X$ is conjunctive.   \end{lem} 

\begin{proof}  Given $W\subsetneq U$ in $\B$, select a point $u\in U$ that does not belong to $W$.  Using the Wallman assumption, select $V\in \B$ with $U \cup V= X$ and $u\not\in V$.   Then $V\cup W\subsetneq X=  V \cup U$. \end{proof}

The converse is not true.  A conjunctive annular base for the topology of $\R$ need not be a a Wallman base.  This is shown by the second example below.  The first example does not accomplish this goal, but we include it because it is useful in understanding the second

\begin{exmp} In this example, let $\B$ be the sublattice of the topology of $\R$ containing $\R$ and $\emptyset$ and generated by the finite open intervals as well as the sets of the form $E\cup U_0$, where $E=(a, \infty)$ or $E=(-\infty, b)$ for some $a,b\in \R$, and $U_0$ is an open interval containing $0$.  The interval $U= (-1, \infty)$ is in $ \B\setminus m_0$, but there is no $V\in m_0$ such that $V\cup U = \R$.  So this is not a Wallman base. This example falls short of our goal, however, because \textit{$\B$ is not conjunctive.}  Let $W = (-1,0)\cup(0,1)$ and let $U = (-1,1)$.  If $V\cup U = \R$, then $0\in V$, so $V\cup W = \R$.  \end{exmp} 

\begin{exmp}\label{exnw} We shall modify the previous example by demanding that the finite open intervals in $\B$ must either contain $0$ or not have $0$ as an endpoint.  Thus, if $0\not\in W\in \B$, then $0$ has a neighborhood in $\B$ that misses $W$.  

\begin{prop} $\B$ of Example \ref{exnw} is  not Wallman, yet it is conjunctive.\end{prop}

\begin{proof}  As in the first example, the interval $U= (-1, \infty)$ is in $ \B\setminus m_0$, but there is no $V\in m_0$ such that $V\cup U = \R$, so this is not a Wallman base. To show that $\B$ is conjunctive, suppose $W, U\in \B$ and $W\subsetneq U$  ($W$ is a proper subset of $U$).  We shall show there is $V\in \B$ such that $V\cup W\subsetneq \R$ and $V\cup U = \R$.  Suppose $U = (a_1, b_1)\cup(a_2, b_2)\cup \cdots\cup (a_n, b_n)$ and $W = (a'_1, b'_1)\cup(a'_2, b'_2)\cup \cdots\cup (a'_{n'}, b'_{n'})$.  Here we have\begin{enumerate}
\item[$(i)$] $-\infty\leq a_1<b_1\leq a_2<b_2\leq \cdots<b_n\leq \infty$;
\item[$(ii)$] for all $i$, $a_i\not=0$ and $b_i\not= 0$;
\item[$(iii)$] if $a_1 = -\infty$ or $b_n=\infty$, then for some $i$, $a_i<0<b_i$.
\end{enumerate}
The same conditions apply to $W$.  Moreover, each interval in the decomposition of $W$ is contained in some interval in the decomposition of $U$.  We might, for example, have $(a_i,b_i)\supseteq(a'_j,b'_j)\cup(a'_{j+1},b'_{j+1})$, with $a_i = a'_j$, $b'_j= a'_{j+1}$ and $b_i = b'_{j+1}$.  In other words $(a_i, b_i)\setminus W$ is a single point, $b'_j$.   Now for any $\epsilon>0$, let $V:=V(\epsilon,U)$, be defined as follows:
\begin{enumerate}
\item $a_1=-\infty$ and $b_n = \infty$.  Then $V$ will be the union $$(b_1-\epsilon, a_2+\epsilon)\cup(b_2-\epsilon, a_3+\epsilon)\cup\cdots\cup(b_{m-1}-\epsilon, a_n+\epsilon).$$
\item $a_1>-\infty$ and $b_n = \infty$.  Then $V$ will be the union $$(-\infty, a_1+\epsilon)\cup(b_1-\epsilon, a_2+\epsilon)\cup(b_2-\epsilon, a_3+\epsilon)\cup\cdots\cup(b_{m-1}-\epsilon, a_n+\epsilon)\cup(-\epsilon,\epsilon).$$
\item $a_1=-\infty$ and $b_n < \infty$.  Then $V$ will be the union $$(b_1-\epsilon, a_2+\epsilon)\cup(b_2-\epsilon, a_3+\epsilon)\cup\cdots\cup(b_{m-1}-\epsilon, a_n+\epsilon), (b_n-\epsilon, \infty)\cup(-\epsilon,\epsilon).$$
\item $a_1>-\infty$ and $b_n < \infty$.  Then $V$ will be the union $$(-\epsilon, \epsilon)\cup(-\infty, a_1+\epsilon)\cup(b_1-\epsilon, a_2+\epsilon)\cup\cdots\cup(b_{m-1}-\epsilon, a_n+\epsilon)\cup (b_n-\epsilon, \infty).$$
\end{enumerate}
Since $W$ is a proper subset of $U$, some of the endpoints $\{a'_j, b'_j\}$ defining the intervals of $W$ are in the interior of the intervals of $U$, or $W$ is the union of a subset of the intervals defining $U$.  Moreover, if $W$ does not contain $0$, then it misses an interval about $0$.  Thus, in any case, we may choose $\epsilon>0$ so small that no endpoints of $W$-intervals that are not endpoints of $U$-intervals belong to $V(\epsilon, U)$.  This assures that $V(\epsilon, U)\cup W\not=\R$.  By construction  $V(\epsilon, U)\cup U=\R$.
\end{proof}

Observe that $\{\,U\subseteq\R\mid \text{$U$ open and $0\not\in U$}\,\}$ is a maximal ideal in the full topology of $\R$.   In contrast, in the example above, $m_0:=\{\,U\in \B\mid 0\not\in U\,\}$ fails to be a maximal ideal of $\B$.    It is contained in $I_{fin}$, the ideal of all finite unions of finite open intervals with non-zero endpoints.   Moreover, $I_{fin}$ itself is not maximal.  It is contained in the ideal $m_{0^+}\subseteq \B$ whose elements are all finite unions of intervals of the form $(a, b)$ where $0\not=a\in \{-\infty\}\cup\R$ and $0\not=b\in \R$, with the proviso that any element of $m_{0^+}$ containing an interval of the form $(-\infty, b)$ must also contain an interval of the form $(-\epsilon, \epsilon)$, $\epsilon>0$.  We also have the ideal $m_{0^-}$, which is defined analogously.   Both these ideals are maximal.   What we have here, then, can be viewed as the real line with two additional points added, call them $0^+$ and $0^-$.  Since $m_0\subseteq m_{0^+}$ and $m_0\subseteq m_{0^-}$, each of the new points belongs to the closure of $0$.  A neighborhood base for $0^+$ consists of the elements in the complement of $m_{0^+}$, namely all finite unions of intervals that contain a set of the form $(-\epsilon, \epsilon)\cup (x,\infty)$, and analogously for $0^-$.  Each has a neighborhood that does not contain the other, but they do not have disjoint neighborhoods.\end{exmp}

 \begin{prop}[\cite{Jbook} IV.2.4] If $\B$ is a Wallman base for a $T_1$-space $X$, then $\eta_B: X\to \SpecM \B$ is an embedding with dense image.
\end{prop}

In the proposition as presented in \cite{Jbook}, $\SpecM \B$ refers to the subspace the prime spectrum $\Spec \B$ consisting of the maximal ideals.  If we treat $\B$ as a join semilattice and construct $\SpecM \B$ as above, then we get the same space.   The reason is that the definition of ideal we used in the context of join semilattices is the same as the definition used for distributive lattices.  The same is true of maximal ideals, and in both settings the definition of the topology is the same.

\subsection{Distributive lattices generated by join semilattices.}\label{join-generation}

For any join-semilattice $L$ with $1$, there is a $1$-$\vee$-preserving morphism $d_L:L\to dL$ that is universal to distributive lattices with $1$.  In this subsection, we consider this construction when $L$ is conjunctive.  By the universal mapping property, there is a $1$-$\vee$-$\wedge$-preserving surjection $\overline{w_L}:dL\to wL$, where $w_L$ is the distributive sub-lattice of the topology of $\SpecM L$ that is generated by $\{\,\coz a\mid a\in L\,\}$.  We provide an example showing that $\overline{w_L}$ may fail to be injective, and we prove that in general $wL \cong dL/R^1$.

\textit{Throughout the remainder of this section, all join-semilattices and all distributive lattices have $1$, and all morphisms preserve $1$.}   Let $L$ be a join semilattice.  We say that $dL$ is a \textit{free distributive lattice over $L$} if $dL$ is a distributive lattice, and there is a $1$-$\vee$-morphism $d_L:L\to dL$ such that for any  $1$-$\vee$-morphism $f:L\to B$, with $B$ a distributive lattice, there is a unique $1$-$\vee$-$\wedge$-morphism $\overline{f}:dL\to B$ such that $f = \overline{f}\,d_L$.  The universal mapping property guarantees uniqueness up to unique isomorphism, so we speak of \textit{the} free distributive lattice with $1$ over $L$.  The existence of $dL$ follows from the fact that distributive lattices form a varietal category.  In order to examine examples, we provide a concrete description.

We may construct $dL$ by imitating the construction in \cite{Jbook}.  Let $\mathcal{P}L$ denote the power set of $L$.  Define $\delta_L: L\to \mathcal{P}L$ by $\delta_L(a) := \up a$.  Let $dL$ to be the sublattice of $\mathcal{P}L$ genertated by the image of $L$.  We  endow $dL$ with the opposite of the natural order, so for subsets $U, V\subseteq L$, $U\leq V$ means $V\subseteq U$,  $U\vee V := U\cap V$ and $U\wedge V := U\cup V$.   (In other words, we are viewing $dL$ as a sub-poset of $(\mathcal{P}L)^{op}$.)  Then $d_L$ is a $1$-$\vee$ morphism.   A slight modification of the proof of   \cite{Jbook}, II.1.2 shows that $\delta_L$ has the universal mapping property required in the definition of the free distributive lattice over $L$.

Below, we shall use  an equivalent representation $dL$ as a sublattice of $\mathcal{P}L$ with the natural order, $U\leq V\;\Leftrightarrow\;U\subseteq V$.  We simply define $d_L(b)$, $b\in L$, to be the complement of $\up b$, i.e.,
\begin{equation}\label{dLdef}d_L(b):= \{\, y\in L\mid b\not\leq  y\,\}\in \mathcal{P}L.\end{equation}  If $L$ is conjunctive, $b\not\leq y$ if and only if $y$ is contained in a maximal ideal missing $b$; therefore, $d_L(b)$  is the set of all $y\in L$ such that $\widehat y$ vanishes at some element of $\coz b$.  The join in $dL$ is set-theoretic union in $\mathcal{P}L$.   As required, $d_L$ is join-preserving:
\begin{equation}d_L(a)\vee d_L(b) = \{\, y\in L\mid a\not\leq  y\;\text{or}\; b\not\leq  y,\}= \{\, y\in L\mid a\vee b\not\leq  y\,\}.\end{equation}
Meets in $dL$ are described as follows: 
\begin{align}\label{meet}d_L(b_1)\wedge\cdots\wedge d_L(b_n) &= \{\,y\in L\mid b_1\not\leq y\,\}\cap\cdots\cap \{\,y\in L\mid b_n\not\leq y\,\} \\
&=\{\,y\in L\mid b_1\not\leq y\;\&\;\cdots\;\&\;b_n\not\leq y\,\}
\end{align} 
If we assume $L$ is conjunctive,  $d_L(b_1)\wedge\cdots\wedge d_L(b_n)$ is the set of $y\in L$ such that $\widehat y$ vanishes at some point of $\coz b_i$ for each $i=1,\ldots, n$. Note that it is not required that $\widehat y$ vanish at some point of the intersection of all these cozero sets.

Now, let $L$ be a conjunctive join semilattice and let $wL$ denote the sublattice of the topology of $\SpecM L$ that is generated by $\{\,\coz a\mid a\in L\,\}$.  For compatibility with other notation, we used the notation $w_L:L\to wL$, with $w_L(a) := \coz a$.
Thus,
\begin{align}w_L(b_1)\wedge\cdots\wedge w_L(b_n) &=  \coz b_1\cap\cdots\cap  \coz b_n \\
&=\{\,m\in \SpecM\mid b_1\not\in m\;\&\cdots\&\;b_n\not\in m\,\}.
\end{align} 

By the universal mapping property of $d_L$, there is a unique  $1$-$\vee$-$\wedge$-morphism 
$\overline{w_L}:dL\to wL$ such that 
$\overline{w_L}\,d_L = w_L$.   Evidently, we must have 
$$\overline{w_L}\,\bigl(\{\,y\in L\mid b_1\not\leq y\;\&\;\cdots\;\&\;b_n\not\leq y\,\}\bigr) = \{\,m\mid b_1\not\in m\;\&\cdots\&\;b_n\not\in m\,\}.$$   This need not be an injective map, as we show below by example. 

\begin{exmp}Let $L$ be any set containing $1$, with join operation defined by $a\vee b = 1$ for any distinct elements $a$ and $b$ of $L$.   This is clearly conjunctive.  We have 
$d_L(b)= L\setminus\{1,b\}$, and 
$d_L(b_1)\wedge \cdots \wedge d_L(b_n) = L\setminus \{1, b_1,\ldots b_n\}$.   The maximal ideals of $L$ are the singletons $\{b\}$, with $b\not=1$.  We have $w_L(b)  = \coz b = \{\{x\}\in \SpecM L\mid x\not= b\,\}$.  The map $\overline{w_L}$ is given by 
$$\overline{w_L}\bigl(L\setminus\{1, b_1,\ldots b_n\}\bigr) = \bigl(\SpecM L\bigr)\setminus\{\,\{b_1\},\ldots,\{b_n\}\,\}.$$  In this example, $\overline{w_L}$ is injective.
\end{exmp}

Let $L$ be a finite conjunctive join semilattice, and let  $X=\SpecM L$.  We view $L$ as a sub-join-semilattice of $\mathcal{P}X$.  Since $X$ is finite and it's topology is $T_1$, for each $x\in X$, $\{x\}^c=X\setminus \{x\} \in L$. The maximal ideal $x$  is the downset of  $\{x\}^c$.  We have that $\mathcal{W}L = wL =\mathcal{P} X$.  
 
 \begin{exmp} Suppose $X=\{a,b,c\}$ and $L$ is the join semilattice that consists of the non-empty elements of $\mathcal{P} X$.  We write the subsets of $X$ as strings: $a := \{a\}$, $ab:=\{a,b\}$, etc.     Thus, $L = \{a,b,c,ab,ac, bc, abc\}$, and  $m_a:=\down bc=\{b,c,bc\}$.   
The following diagrams show the structure of $dL$.  To avoid clutter, we use $i$, $j$, and $k$ to represent $a$, $b$ and $c$ taken in any order.  The diagram on the left represents the elements as meets of the generators, while the one on the right shows the corresponding elements as downsets of $L$.   Each node on the right labeled with $i$, $j$, and/or $k$ represents three distinct elements of $dL$.  Accordingly, the cardinality of $d_L$ is 18.  Since the cardinality of $wL$ is $8$, evidently, $\overline{w_L}$ is not injective.

\bigskip
{\scriptsize
\hfill
\begin{tikzcd}[tips=false,column sep=.1em,row sep=1em]
    &d_L(ijk)\ar{d}\\
    &d_L(ij)\ar{d}\\
    &d_L(ij)\wedge d_L(ik)\ar{dl} \ar{dr}\\
    d_L(i)\ar{dr} &&d_L(ij)\wedge d_L(ik)\wedge   d_L(jk)\ar{dl}\\
    &d_L(i)\wedge d_L(jk)\ar{d}\\
    &d_L(i)\wedge d_L(j)\ar{d}\\
   & \emptyset
\end{tikzcd}
\hfill
\begin{tikzcd}[tips=false,column sep=.1em,row sep=1em]
    &\{ab, ac, bc, a,b,c\}\ar{d}\\
    &\{ik,jk,i,j,k\}\ar{d}\\
    &\{jk,i,j,k\}\ar{dl} \ar{dr}\\
    \{jk, j,k,\}\ar{dr} &&\{a,b,c\}\ar{dl}\\
    &\{j,k\}\ar{d}\\
    &\{k\}\ar{d}\\
   & \emptyset
\end{tikzcd}
\hfill
}
\end{exmp}

\begin{prop}  Let $L$ be a conjunctive join-semilattice.   The congruence on $dL$ induced by $\overline{w_L}:dL\to wL$ is $R^1(dL)$, and hence $wL\cong dL/R^1$.  \end{prop}
\begin{proof}  First, by the same arguments used in the proof of Lemma \ref{replemma}, we have that  $\SpecM wL = \SpecM L$.  By Lemma \ref{conjmax}, $wL$ is conjunctive.  We now show that $\overline{w_L}^{-1}(1_{WL}) =\{1_{dL}\}$.  Suppose $q\in dL\setminus\{1\}$.  We show that $\overline{w_L}(q) \not= 1$.  Since $d_L:L\to dL$ is a $1$-$\vee$-morphism and $dL$ is distributive,  $q =\bigvee_{i=1}^n d_L(b_i)$, with $b_i\in L\setminus \{1\}$.  Thus, $q\leq d_L(b_i)$, so $\overline{w_L}(q)\leq  \overline{w_L}d_L(b_i) = \coz b_i<1$.  In light of Proposition \ref{distpierce}, Lemma \ref{piercelemma} applies to distributive lattices, finishing the proof.  \end{proof} 

Of course, since $wL$ is a distributive lattice with $1$, all the elements of $\SpecM wL$ are prime, and $\SpecM wL$ is simply the subspace of the classical spectral space $\Spec wL$ consisting of the maximal ideals.   If $L$ is a conjunctive distributive lattice, then  $\SpecM L$ is dense in $\Spec L$, so in this case, $L = wL$.

The following example shows that $w$ is not functorial.  Let  $L=\mathcal P\{a,b,c\} = \{0=\emptyset, a,b,c,ab,ac,bc,1=abc\}$ and let $M=\mathcal P\{d\} = \{0,1\}$.  Define $\phi:L\to M$ by $\phi(0) = \phi(a) = 0$ and $\phi(x)=1$ for $x\in L\setminus \{0,a\}$.  This is indeed a $\vee$-morphism, because $\{0,a\}$ is an ideal.   In this case, $wL = L$ and $wM=M$, so $w\phi$, if it existed, would be $\phi$.  But $\phi$ is not a $\wedge$-morphism, because $ab\wedge ac = a$, though $\phi(ab)\wedge \phi(ac) = 1$ and $\phi(a) = 0$.

\end{document}